\documentclass[12pt,a4paper]{article}
\usepackage[T1]{fontenc}
\usepackage[margin=3cm,footskip=1cm]{geometry}

\usepackage{amsmath} 
\usepackage{amssymb,amsthm,bm,mathtools,xspace,booktabs,url}
\usepackage{graphicx,etoolbox}
\usepackage{ dsfont }
\usepackage{hyperref}
\usepackage{bbold}
\usepackage{tikz}
\usepackage{bigints}
\usepackage{mathrsfs}
\usepackage[shortlabels]{enumitem}
\usepackage{xcolor}
\usepackage{color}

\newcommand{\dd}{\mathrm{d}}
\makeatletter
\global\let\tikz@ensure@dollar@catcode=\relax
\makeatother
\setlist{
  listparindent=\parindent,
  parsep=0pt,
}
\usepackage{amssymb}

\AtBeginEnvironment{thebibliography}{
  \small
  \setlength\itemsep{0.75em plus 0.5em minus 0.75em}%
}

\usepackage{caption}
\usepackage[raggedright]{titlesec}

\numberwithin{equation}{section}
\usepackage{amsfonts}

\theoremstyle{plain} 
\newtheorem{theorem}{Theorem}[section]

\newtheorem{definition}[theorem]{Definition}

\theoremstyle{definition} 

\newtheorem{Remark}[theorem]{Remark}

\usepackage{authblk}

\makeatletter
\newcommand\CorrespondingAuthor[1]{
  \begingroup
  \def\@makefnmark{}
  \footnotetext{Corresponding author: #1}
  \endgroup
}
\makeatother

\makeatletter
\renewenvironment{abstract}{%
  \small%
  \providecommand\keywords{%
    \par\medskip\noindent\textit{Keywords:}\xspace}%
  \begin{center}%
    \bfseries \abstractname\vspace{-.5em}\vspace{\z@}%
  \end{center}%
  \quote%
}
{
\endquote}
\makeatother

\usepackage[above,below,section]{placeins}

\usepackage[all]{onlyamsmath}

\usepackage{multirow}

\usepackage{xcolor}
\usepackage{color}

\usepackage{mathrsfs}
\usepackage{pdfpages}

\definecolor{darkmagenta}{rgb}{0.5,0,0.5}
\definecolor{darkgreen}{rgb}{0,0.6,0}
\definecolor{darkblue}{rgb}{0,0,0.6}
\definecolor{darkred}{rgb}{0.8,0,0}
\definecolor{mellow}{rgb}{.847, 0.72, 0.525}
\newcommand{\rouge}[1]{\textcolor{darkred}{#1}}

\usepackage[normalem]{ulem}

\newcommand{\mce}{\mathcal{E}}

\begin{document}

\title{Strong convergence to two-dimensional alternating Brownian motion processes}

\author[1]{Guy Latouche}
\affil[1]{Department of Computer Science, Universit\'{e} Libre de Bruxelles}
\author[2]{Giang T. Nguyen}
\author[2]{Oscar Peralta}
\affil[2]{School of Mathematical Sciences, The University of Adelaide}
\date{\empty}
\maketitle

\begin{abstract}
{Flip-flop processes refer to a family of stochastic fluid
  processes which converge to either a standard Brownian motion (SBM)
  or to a Markov modulated Brownian motion (MMBM). In recent years, it
  has been shown that complex distributional aspects of the univariate
  SBM and MMBM can be studied through the limiting behaviour of
  flip-flop processes. Here, we construct two classes of bivariate
  flip-flop processes whose marginals converge strongly to SBMs and
  are dependent on each other, which we refer to as \emph{alternating
    two-dimensional Brownian motion processes}. While the limiting
  bivariate processes are not Gaussian, they possess desirable
  qualities, such as being tractable and having a time-varying
  correlation coefficient function.}

\keywords{ Strong convergence, two-dimensional Brownian motion, stochastic fluid processes, flip-flops, time-varying correlation coefficient} 

\emph{2010 Mathematics Subject Classification:} Primary: 60F15, 60G50; Secondary: 60G15

\end{abstract}

\section{Introduction}
{Let $\mathcal{B} = \{B(t)\}_{t\ge 0}$ denote a standard Brownian motion. The construction of discrete-time random walk approximations to $\mathcal{B}$ is a well-known problem initiated by Donsker in \cite{donsker1951}. His work dealt with weak approximations, that is, with the convergence of the \emph{laws} of appropriately scaled random walks to the law of $\mathcal{B}$. A different technique using strong approximations was proposed by Strassen \cite{strassen1964invariance}, which involved the construction of a probability space where the paths of the scaled random walks were shown to converge uniformly on compact intervals to the path of $\mathcal{B}$. Not only did Strassen's technique provide an alternative approach to the construction of approximations to $\mathcal{B}$, but it also yielded stronger results: a strongly convergent family of stochastic processes to $\mathcal{B}$ is automatically weakly convergent, while the opposite is not always true. Furthermore, strong convergence allows one to claim convergence of pathwise functionals of the processes.}

In this paper we focus on strong approximations to $\mathcal{B}$ and related models, not using discrete-time random walks, but a family of continuous-time stochastic processes which we introduce next.
\begin{definition} 
 {Consider} a family of processes $\{\mathcal{F}_\lambda\}_{\lambda > 0}=\{\{F_\lambda(t)\}_{t\ge 0}\}_{\lambda > 0}$, where
\begin{align*}
F_\lambda(t) := \sqrt{\lambda}\int_0^t J_\lambda(s)\dd s,\quad t\ge 0,
\end{align*}
and the \emph{phase} process $\mathcal{J}_\lambda = \{J_\lambda(t)\}_{t\ge 0}$ denotes a Markov jump process with state space $\mathcal{E} =\{1,-1\}$ and intensity matrix given by 
\begin{align}
   \label{e:Lam}
\Lambda= \left[\begin{array}{rr}
-\lambda&\lambda\\
\lambda&-\lambda
\end{array}\right]. 
\end{align}
We refer to $\mathcal{F}_\lambda$ as the \emph{standard flip-flop process}.
\end{definition} 
\noindent {Intuitively,} the process $\mathcal{F}_\lambda$ describes the path of a particle in $\mathds{R}$ with finite constant velocity which randomly switches direction at a constant rate. Such a process has been studied in the literature as the \emph{uniform transport process} \cite{watanabe1968approximation,griego1971almost}, and {is also a special case} of the \emph{telegraph process} \cite{goldstein1951,kac1974} and, {more recently, of} the \emph{flip-flop process} \cite{ramaswami2013fluid,latouche2015morphing}.

{The standard flip-flop process $\mathcal{F}_\lambda$ belongs to a larger class of piecewise-linear models modulated by a Markov jump process, called \emph{stochastic fluid processes} (SFP), for which several first passage results and formulae are available \cite{latouche2018analysis}. Even if the behaviour of $\mathcal{F}_\lambda$ is somewhat similar to that of a random walk with double-sided exponential increments, notice that changes in direction of $\mathcal{F}^\lambda$ happen at random time epochs. In comparison, the discrete-time random walks approximations to $\mathcal{B}$ devised in \cite{donsker1951, strassen1964invariance} have inflection points on a deterministic grid. Though this characteristic of $\mathcal{F}_\lambda$ may seem inconvenient at first, once a limiting result is established, it actually allows us to translate well-understood results from SFPs into complex second-order models such as the Markov modulated Brownian motion (MMBM), as recently evidenced in \cite{latouche2015morphing,latouche2015fluid,ahn2017quadratically,latouche2017slowing,ahn2017time,latouche2018markov}. Besides being important objects of study within the matrix-analytic-methods community, both SFPs and MMBMs have been succesfully used to model random systems in queueing theory, risk theory, finance and epidemiology.}


{The \emph{weak} convergence of $\mathcal{F}_\lambda$ to $\mathcal{B}$ as $\lambda \rightarrow \infty$ was first established in \cite{watanabe1968approximation}. Through different constructions, it was shown in \cite{griego1971almost,nguyen2019strong,nguyen2020explicit} that $\mathcal{F}_{\lambda}$ also converges \emph{strongly} to $\mathcal{B}$; that is,  
they showed the existence of a probability space $(\Omega,\mathscr{F}, \mathds{P})$ such that for all $T\ge 0$,
\begin{align*}
	\lim_{\lambda\rightarrow\infty}\sup_{0\le s\le T} |F_\lambda(s)-B(s)| = 0 \quad \mbox{almost surely}.
\end{align*}}
\indent In order to extend these strong approximations into two-dimensional processes, consider two families of standard flip-flops, $\{\mathcal{F}^1_{\lambda}\}$ and $\{\mathcal{F}^2_{\lambda}\}$, with corresponding families of phase processes $\{\mathcal{J}_{\lambda}^1\}$ and $\{\mathcal{J}_{\lambda}^2\}$, living on two probability spaces $(\Omega^1,\mathscr{F}^1, \mathds{P}^1)$ and $(\Omega^2,\mathscr{F}^2, \mathds{P}^2)$, respectively. One can see that, on $(\Omega^1\times\Omega^2,\mathscr{F}^1\otimes\mathscr{F}^2, \mathds{P}^1\times\mathds{P}^2)$, the bivariate process $(\mathcal{F}^1_{\lambda},\mathcal{F}^2_{\lambda})$ converges strongly to a two-dimensional Brownian motion $(\mathcal{B}^1, \mathcal{B}^2)$ as $\lambda\rightarrow\infty$. 
%

The process $\left(\mathcal{F}^1_\lambda, \mathcal{F}^2_\lambda\right)$ can be redefined such that $\mathcal{F}^1_\lambda$ and $\mathcal{F}^2_\lambda$ are jointly modulated by a common phase process, as follows. 

\begin{definition} {Let 
$\boldsymbol{\mathcal{J}}_{\lambda} = \{\boldsymbol{J}_{\lambda}(t)\} = \left\{\left(J_{\lambda}^1(t), J_{\lambda}^2(t)\right)\right\}$ be} a Markov jump process on state space $\mathcal{E}^2 = \{(1,1),(1,-1), (-1,1), (-1,-1)\}$ with the intensity
matrix
\begin{align}
	\label{eq:indgen1}
	\Lambda\oplus\Lambda = \left[\begin{array}{rrrr}
		-2\lambda & \lambda & \lambda & 0\\ \lambda & -2\lambda & 0 & \lambda\\ \lambda & 0 &-2\lambda & \lambda\\0&\lambda & \lambda & -2\lambda
			\end{array}\right],
\end{align}
where $\oplus$ denotes the Kronecker sum. Let \begin{align*}
F^i_\lambda(t) := \sqrt{\lambda}\int_0^t \pi_i\left(\boldsymbol{J}_\lambda(s)\right)\dd s, \quad t \ge 0, i = 1, 2, 
\end{align*}
where $\pi_i:\mathcal{E}^2 \mapsto \mathcal{E}$ denotes the $i$th coordinate projection. We refer to $\left(\mathcal{F}^1_{\lambda},\mathcal{F}^2_{\lambda}\right)$ as a \emph{standard bivariate flip-flop} process, and $\boldsymbol{\mathcal{J}}_\lambda$ its underlying phase process. 
\end{definition} 
{More broadly, we define the following.
\begin{definition}
A \emph{bivariate flip-flop} process
$\left(\mathcal{G}^1_\lambda,\mathcal{G}^2_\lambda\right) = \left\{\left({G}^1_\lambda(t),{G}^2_\lambda(t)\right)\right\}_{t\ge
  0}$ is defined by
\begin{align}
   \label{e:gi}
G^i_\lambda(t) := \sqrt{\lambda}\int_0^t \pi_i^*(\bm{J}_\lambda(s))\dd s,\quad t\ge 0,
\end{align}
where $\bm{\mathcal{J}}_\lambda$ is a Markov jump process on $\mathcal{E}^2\times\mathcal{S}$, where $\mathcal{S}$ is some finite collection of states, and $\pi_i^*: \mathcal{E}^2\times\mathcal{S}\rightarrow \mathds{R}$, $i\in\{1,2\},$ is such that 
\begin{align*} 
	\pi_1^*((1,\mathcal{E})\times\mathcal{S}) & =\pi_2^*((\mathcal{E},1)\times\mathcal{S})=1, \\
	 \pi_1^*((-1,\mathcal{E})\times\mathcal{S}) & =\pi_2^*((\mathcal{E},-1)\times\mathcal{S})=-1.
\end{align*} 
\end{definition} }

{Thus, a bivariate flip-flop process has an underlying phase process $\boldsymbol{\mathcal{J}}_\lambda$ with a considerably richer structure than that of (\ref{eq:indgen1}), which may be used to produce dependent coordinate processes instead of independent ones. One natural goal is to characterise all possible strong limits for bivariate flip-flop processes and study their dependence structures.} 

As a step toward that goal, we construct in this paper two families of bivariate flip-flop processes and show that each converges strongly
to a process, of which the marginals are standard Brownian motion and
the correlation between the two components is a time-varying function. In short, the marginals of the
limiting process alternate between \emph{synchronising} intervals
during which they evolve with identical increments to each other, and
\emph{desynchronising} intervals during which they evolve as mirror images of each other.  In Section \ref{sec:exprefl} the intervals are
exponentially distributed and we call the limiting process a \emph{two-dimensional exponentially alternating Brownian motion}. In Section \ref{sec:MAP} the
synchronisation and desynchronisation epochs are determined by a
continuous-time Markovian arrival process (MAP); we refer to this process as a \emph{two-dimensional  MAP alternating Brownian motion}.
We determine the time-varying correlation functions of the limiting
processes in Sections \ref{sec:exprefl} and \ref{sec:MAP},
respectively.

From a different perspective, the bivariate flip-flop processes can be
seen as a continuous-time analogue of the discrete-time
\emph{bootstrapping random walk} concept proposed in
\cite{Collevecchio2016, Collevecchio2019}. There, the authors start by
constructing a discrete-time random walk from Bernoulli random
variables, and then reuse (bootstrap) these random variables to create
additional walks; together, the random walks converge weakly
to two- or higher-dimensional Brownian motions. In our case, we will
start by constructing a continuous-time random walk (standard
flip-flop) from exponential random variables, and then reuse them to
create an additional continuous-time random walk; together, the random
walks (bivariate flip-flop) converge strongly to a two-dimensional
process.  {In one construction, the strong limit is a two-dimensional Brownian motion; in two others, the strong limits are two-dimensional alternating Brownian motion processes. 

{Another related model is the \emph{zig-zag process}~\cite{bierkens2016zig}, a multidimensional piecewise deterministic Markov process which behaves in a piecewise linear fashion over time. Changes in direction of the path of a zig-zag process occur according to some space-dependent intensity function, the latter chosen in such a way that the stationary behaviour of the process follows some prespecified distribution. Although flip-flop processes are in fact a space-homogeneous version of the zig-zag processes, their use differ: zig-zag processes are commonly used to implement efficient sampling methods \cite{bierkens2017limit,bierkens2019zig}, while the goal of flip-flop processes is to explain properties of complex second-order processes.}
} 

\section{Preliminaries} 
	\label{sec:prelim}

Our construction of bivariate flip-flops is inspired by
{\cite{nguyen2020explicit}}, and we sketch in this section the construction proposed there for one-dimensional standard flip-flops. The creation in {\cite{nguyen2020explicit}}  of a family of standard flip-flops $(\mathcal{F}_{\lambda}, \mathcal{J}_{\lambda})$, with $\lambda \rightarrow\infty$, relies on inspecting the standard Brownian motion $\mathcal{B}$ at the arrival epochs of a Poisson process, as explained next.

{For $\lambda>0$, let $\{\tau^{\lambda}_k\}_{k\ge 0}$, with $\tau_0^\lambda=0$, be the
arrival epochs of a Poisson process of intensity $2\lambda$, independent of $\mathcal{B}$}. For notational convenience, we omit the superscript~$\lambda$ when no ambiguity might arise. 
{Let 
	\begin{align*} 
	C_k := B(\tau_{k+1})-B(\tau_{k}),\quad k\ge 0;
	\end{align*} 
  the sequence $\{C_k\}_{k\ge 0}$ has i.i.d. elements with common double exponential distribution of parameter $2\sqrt{\lambda}$ \cite[Theorem 2.1]{nguyen2020explicit}, that is, following the density function 
  \[f(x)=\sqrt{\lambda}\, e^{-2\sqrt{\lambda}\,|x|},\quad x \in\mathds{R}.\]
  One can verify that $\{|C_k|\}_{k \geq 0}$ and $\{C_k/|C_k|\}_{k \geq 0}$ are independent i.i.d. sequences where $|C_k|\sim \mbox{Exp}(2\sqrt{\lambda})$ and
$
  \mathds{P}\left(C_k/|C_k| =1\right)=\mathds{P}\left(C_k/|C_k| =-1\right)=1/2.
$}

	Next, define the epochs {$\{\xi_k^\lambda\}_{k\ge 0}$}, which we write {$\{\xi_k\}_{k\ge 0}$}, with $\xi_0=0$ by
        \[{\xi_{k+1}-\xi_{k} :=\lambda^{-1/2} |C_k|,}\]
        and the process $\mathcal{J}_\lambda=\{J_\lambda (t)\}_{t\ge 0}$ by
        \begin{align}
        		\label{eqn:Jlambda} 
		J_\lambda(t) := C_k/ |C_k|, \quad \mbox{ for } t\in [\xi_k,\xi_{k+1}), \,k\ge 0.
	\end{align} 
{Note that $\mathcal{J}_\lambda$ might or might not change states at a
  given epoch $\xi_k$. Let $\{\chi_k\}_{k\ge 1}$, 
$\chi_0:=0$, denote the jump times of $\mathcal{J}_\lambda$. By the independence properties of $\{|C_k|\}_{k\le 0}$ and $\{C_k/|C_k|\}_{k\le 0}$, the elements of $\{\chi_{k+1}-\chi_k\}_{k\ge 0}$ are i.i.d. random variables which follow a $\mbox{Geometric}(0.5)$ random sum of exponentials of intensity $2\lambda$. Equivalently, $\chi_{k + 1} - \chi_k$ is an exponential random variable of intensity $\lambda$. Thus, $\mathcal{J}_\lambda$ is a Markov jump process driven by the rate matrix $\left[\begin{smallmatrix*}[r] -\lambda & \lambda\\ \lambda &-\lambda\end{smallmatrix*}\right]$ and has initial distribution $(1/2,1/2)$.
}
    
{Intuitively, to see that $\mathcal{F}_{\lambda}=\{F_\lambda (t)\}_{t\ge 0}$, defined by 
\begin{align}
	\label{eq:strongF5}
	F_\lambda(t) = \sqrt{\lambda}\int_0^t J_\lambda(s)\dd s,
\end{align} 
approximates $\mathcal{B}$ on $[0, T$], note that 
$F_\lambda (\xi_k)=B(\tau_k)$ for $k\ge 0$. Moreover, since $\tau_k \sim \mbox{Erlang}(k,2 \lambda)$ and $\xi_k \sim \mbox{Erlang}(k,2 \lambda)$, we have $\mathds{E}(\tau_k)=\mathds{E}(\xi_k)$, implying that on average the epochs $\{\tau_k\}$ coincide with $\{\xi_k\}$. More precisely, Nguyen and Peralta~\cite{nguyen2020explicit} showed that 
\begin{align*} 
\lim_{\lambda \rightarrow\infty}\max_{k: \tau_k^\lambda,\chi_k^\lambda < T} |\tau_k^\lambda - \xi_k^\lambda| & = 0\quad \mbox{almost surely,}
\end{align*} 
and furthermore,
\begin{align}\label{eq:strongF}\lim_{\lambda\rightarrow\infty}\sup_{0\le s\le T} |F_\lambda(s)-B(s)| = 0 \quad \mbox{almost surely}.
\end{align}
See Figure \ref{fig:standardflipflop} for a sample path approximation of the standard Brownian motion $\mathcal{B}$ with the flip-flop process $\mathcal{F}_\lambda$.
}
    \begin{figure}[h!]
	\begin{center} 
	\hspace*{-0.8cm}
	\includegraphics[scale=0.45]{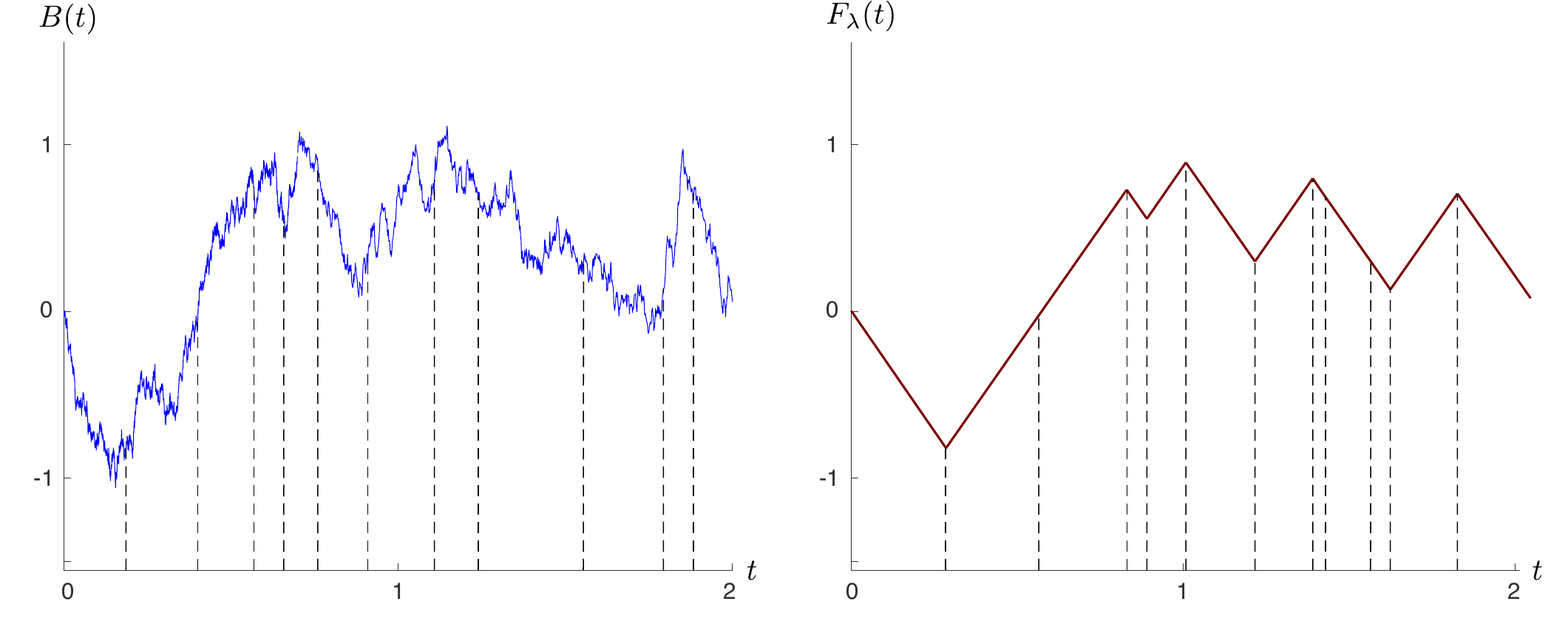}
	%
	\end{center} 
\caption{(Left) {A sample path of an SBM $\mathcal{B}$. The dashed lines correspond to the arrival epochs $\{\tau_k\}_{k\geq 0}$.} (Right) {An associated sample path of a standard flip-flop
  $\mathcal{F}_{\lambda}$. The epochs $\{\xi_k\}_{k\geq 0}$ are shown with dashed lines. Note the values of $F_{\lambda}(\xi_k)$ match
  those of $B(\tau_k)$.}}
\label{fig:standardflipflop}
\end{figure}

 
\section{Reflection at exponential times}
	\label{sec:exprefl}
Here, we construct on a probability space $(\Omega, \mathscr{F}, \mathds{P})$ a family of bivariate flip-flops {$\{(\mathcal{F}^1_{\lambda},\mathcal{F}^2_{\lambda})\}_{\lambda \ge 0}$}, such that {$\{(\mathcal{F}^1_{\lambda},\mathcal{F}^2_{\lambda})\}$} converges strongly to a so-called \emph{two-dimensional exponentially alternating Brownian motion} as $\lambda \rightarrow \infty$. 	
	
\begin{definition} 
	\label{defn:eJH} 
	Let $\mathcal{X} = \{X(t)\}_{t \geq 0}$ be a Markov jump
        process with state space $\{0,1\}$, initial probability $(1,0)$ and intensity matrix 
	\begin{align} 
	\label{eqn:eJH}
		\Pi = \left[\begin{array}{rr} 
		-\alpha& \alpha \\
		\beta & -\beta 
		\end{array}\right], 
	\end{align}  
	and let $\mathcal{B}$ be a standard Brownian motion. Define $\mathcal{B}^* = \{B^*(t)\}_{t \geq 0}$ such that 
	\begin{align}\label{eq:defBstar} 
		B^*(t) = \int_0^t(-1)^{X(s)}\dd B(s), \quad t \geq 0. 
	\end{align} 
The process $(\mathcal{B}, \mathcal{B}^*)$ is called a two-dimensional exponentially alternating Brownian motion.
\end{definition} 
\noindent In the intervals during which $X(t) = 0$, the processes $\mathcal{B}$ and $\mathcal{B}^*$ evolve with precisely the same increment; we refer to these periods as \emph{synchronising intervals}. In intervals during which $X(t) = 1$, the increment of one process is the negative of the other; these are \emph{desynchronising intervals}. 

{It is important to notice that $(\mathcal{B}, \mathcal{B}^*)$ is not a Gaussian process: for instance, the distribution of $B(t)-B^*(t)$ has a point mass at $0$ of size $e^{-\alpha t}$, corresponding to the event $\{X(s)=0 \mbox{ for all } 0\le s\le t\}$. Despite not being Gaussian, $(\mathcal{B}, \mathcal{B}^*)$ is still a process with tractable properties, as we show later in Section \ref{sec:dependence1}.}
 
\subsection{Construction} 
	\label{subsec:constructexpo}
  {Let $\mathcal{J}_\lambda^1 = \mathcal{J}_\lambda$ and define $\mathcal{J}_\lambda^2=\{J_\lambda^2 (t)\}_{t\ge 0}$ by}
  \[ {J_\lambda^2 (t) := (-1)^{X(t)}J_\lambda (t),\quad t\ge 0,}\]
{where $\mathcal{J}_\lambda=\{J_\lambda (t)\}_{t\ge 0}$ is the Markov jump process constructed in \eqref{eqn:Jlambda}. Then $(\mathcal{J}_\lambda^1,\mathcal{J}_\lambda^2)$ is a Markov jump process with state space \[\mathcal{E}^2=\{(1,1),(1,-1), (-1,1), (-1,-1)\},\] initial distribution $(1/2,0,0,1/2)$, and has the intensity matrix given by
\begin{align}
\Pi_\lambda=
  \left[\begin{array}{cccc} 
    -\lambda-\alpha  & \alpha & 0&  \lambda\\
\beta&-\lambda-\beta&\lambda&0\\
0 & \lambda&-\lambda -\alpha& \alpha\\
\lambda&0&\beta& -\lambda-\beta
  \end{array}\right].    \label{eq:bivintmatr2}
\end{align}
Indeed, for $j\in\{1,-1\}$, jumps from $(j,1)$ to $(-j,-1)$ are linked to those jumps associated to $\mathcal{J}_\lambda$ which happen with rate $\lambda$, while jumps from $(j,1)$ to $(j,-1)$ occur thanks to desynchronisation events which happen with rate $\alpha$. Similar explanations hold for jumps originating from $(j,-1)$.
}

Define the flip-flop processes $\mathcal{F}^{1}_\lambda=\{F_\lambda^1 (t)\}_{t\ge 0}$ and $\mathcal{F}^2_\lambda=\{F_\lambda^2 (t)\}_{t\ge 0}$ by
\[F^{i}_\lambda(t) :=\sqrt{\lambda}\int_0^t J^i_\lambda (s)\,\dd s,\quad i\in\{1,2\},\,t\ge 0.\]
\begin{theorem}
  \label{th:relfectingexp1strong}
 {As $\lambda \rightarrow \infty$,} the bivariate flip-flop $(\mathcal{F}^1_{\lambda}, \mathcal{F}^2_{\lambda})$ converges strongly to the process $(\mathcal{B},\mathcal{B}^*)$; that is, for all $T>0$,
  \begin{align}
  \label{eq:strongR2}
\lim_{n\rightarrow\infty} \sup_{0\le s\le T} 
	\left\Vert \left(F^1_{\lambda}(s), F^2_{\lambda}(s)\right) - \left(B(s),B^*(s)\right)\right\Vert_\infty=0,
  \end{align}
where $\Vert \cdot \Vert_\infty$ denotes the max--norm in $\mathds{R}^2$. 

Furthermore, for all $q>0$ there exists $\varepsilon>0$ such that
\begin{align}
	\label{eq:ratestrong1}
	\mathds{P}\left( \sup_{0\le s\le T} 
	\left\Vert\left(F^1_{\lambda}(s), F^2_{\lambda}(s)\right) - \left(B(s),B^*(s)\right)\right\Vert_\infty  \ge  {\eta_{q, \lambda}} \varepsilon (\log \lambda)\lambda^{-1/4}\right) =o(\lambda^{-q}), \end{align}
where $\eta_{q, \lambda} := \max(1, 2q \log \lambda)$. The term $ \eta_{q, \lambda} \varepsilon(\log \lambda)\lambda^{-1/4}$ is said to be the \emph{rate of strong convergence} of $(\mathcal{F}^1_\lambda, \mathcal{F}^2_\lambda)$ to $(\mathcal{B}, \mathcal{B}^*)$.
	
%
\end{theorem}
\begin{proof}
Note that 
\[
  \sup_{0\le s\le T} \left \Vert\left(F^1_{\lambda}(s), F^2_{\lambda}(s)\right) - \left(B(s),B^*(s)\right)\right \Vert_\infty \le \sup_{0\le s\le T} \left| F^1_{\lambda}(s) - B(s)\right| + \sup_{0\le s\le T}\left| F^2_{\lambda}(s) - B^*(s)\right|.
\]
Then, (\ref{eq:strongR2}) follows once we prove that $\mathcal{F}^1_{\lambda}$ strongly converges to $\mathcal{B}$, and that $\mathcal{F}^2_{\lambda}$ strongly converges to $\mathcal{B}^*$. The former is a consequence of $\mathcal{F}^1_{\lambda}=\mathcal{F}_{\lambda}$, with $\mathcal{F}_{\lambda}$ as defined in (\ref{eq:strongF5}), and the strong convergence (\ref{eq:strongF}) proved in \cite{nguyen2020explicit}. In fact, Nguyen and Peralta~\cite{nguyen2020explicit} proved that for all $q>0$ there exists $\varepsilon>0$ such that
\begin{align}
	\label{eq:ratestrong1}
	\mathds{P}\left(\sup_{0\le s\le T}|F_{\lambda}(s) - B(s))|\ge \varepsilon(\log \lambda)\lambda^{-1/4}\right)=o(\lambda^{-q}); \end{align}
the term $\varepsilon(\log \lambda)\lambda^{-1/4}$ {is} the \emph{rate of strong convergence} of $\mathcal{F}_\lambda$ to $\mathcal{B}$.

The case of the convergence of $\mathcal{F}^2_\lambda$ to $\mathcal{B}^*$ is more complicated because the synchronisation/desynchronisation events may cause further separation of the paths, in comparison to that of $\mathcal{F}_\lambda$ to $\mathcal{B}$. More precisely, let $\{\rho_k\}_{k\ge 0}$ be the jump times of $\mathcal{X}$, defined by $\rho_{k+1} := \inf\{s>\rho_k: X(s-)\neq X(s)\}$ with $\rho_0=0$. Now, suppose that $\rho_1\le T<\rho_2$: then,
\begin{align*}
\sup_{\rho_1\le s\le T}|F^2_{\lambda}(s) - B^*(s)| 
& {= \sup_{\rho_1 \leq s \leq  T} \left| F_{\lambda}(\rho_1)  + \left(-F_{\lambda}(s) + F_{\lambda}(\rho_1)\right) - [B(\rho_1) + (-B(s) + B(\rho_1))] \right| } 
\\ 
& {= \sup_{\rho_1 \leq s \leq  T}  \left| 2 F_{\lambda} (\rho_1) - F_{\lambda}(s) - 2B (\rho_1) + B(s)\right|}  
\\ 
& = \sup_{\rho_1\le s\le T}|{2}(F_\lambda(\rho_1)-B(\rho_1)) + (B(s) - F_\lambda(s))|\\
& \le {2} |F_\lambda(\rho_1)-B(\rho_1)| + \sup_{\rho_1\le s\le T}\left|F_{\lambda}(s) - B(s)\right|\\
& \le {3} \sup_{0\le s\le T}\left| F_{\lambda}(s) - B(s)\right|{.}
\end{align*}
implying that
\[\sup_{0\le s\le T} \left| F^2_{\lambda}(s) - B^*(s) \right| \le {3} \sup_{0\le s\le T}\left| F_{\lambda}(s) - B(s) \right|.\]
Using recursive arguments, we get the bound for the general case
\begin{align}
\sup_{0\le s\le T}\left| F^2_{\lambda}(s) - B^*(s)\right| & \le (2R - 1) \left[\sup_{0\le s\le T} \left| F_{\lambda}(s) - B(s)\right| \right],	\label{eq:auxR1}
\end{align}
where $R:= \inf\{k\ge 1: \rho_k>T\} $. Equation \eqref{eq:auxR1} implies that the convergence of $\mathcal{F}^2_\lambda$ to $\mathcal{B}^*$ is \emph{slower} than that of $\mathcal{F}_\lambda$ to $\mathcal{B}$. 

To precisely measure its rate of strong convergence, note that since $R$ has {a} Poissonian {tail} of parameter $\gamma_0 T $ with $\gamma_0:= \alpha\wedge\beta$; {by} \cite{glynn1987upper},
$
\mathds{P}(R> n)= o\left(e^{-n(\log n -\gamma_0T)}\right). 
$
{This} in turn {gives}
\begin{align}
	\label{eq:auxR2}
\mathds{P}(R> q\log\lambda)= o\left(\lambda^{-q}\right).
\end{align}
Equations (\ref{eq:ratestrong1}), (\ref{eq:auxR1}) and (\ref{eq:auxR2}) suggest that the rate of strong convergence of $\mathcal{F}_\lambda^2$ to $\mathcal{B}^*$ is $\varepsilon^*(\log \lambda)^2\lambda^{-1/4}$ where $\varepsilon^* := 2 \varepsilon q$. Indeed,
\begin{align*}
&\mathds{P}\left(\sup_{0\le s\le T}|F^2_{\lambda}(s) - B^*(s)|\ge \varepsilon^*(\log \lambda)^2\lambda^{-1/4}\right)\\
&\quad\le \mathds{P}\left((2R - 1) \left[\sup_{0\le s\le T} \left| {F}_{\lambda}(s) - B(s)\right| \right] \ge \varepsilon^*(\log \lambda)^2\lambda^{-1/4}\right)\\
&\quad\le \mathds{P}\left( (2R - 1) \left[\sup_{0\le s\le T}\left| {F}_{\lambda}(s) - B(s)\right| \right] \ge \varepsilon^*(\log \lambda)^2\lambda^{-1/4}, R \le q\log\lambda\right) \\
& \quad \quad + \mathds{P}(R> q\log\lambda)\\
& \quad\le \mathds{P}\left(\sup_{0\le s\le T}|F^1_{\lambda}(s) - B(s)| \ge \varepsilon(\log \lambda)\lambda^{-1/4}\right) + \mathds{P}(R> q\log\lambda)=o(\lambda^{-q}),
\end{align*}
and, thus, $\sup_{0\le s\le T}\left|F^2_{\lambda}(s) - B^*(s)\right|$ converges almost surely to $0$.

{Finally, the strong convergence rate for the bivariate process $(\mathcal{F}^1_{\lambda}, \mathcal{F}^2_{\lambda})$ follows from taking the maximum between the two strong convergence rates of the marginal processes $\mathcal{F}^1_{\lambda}$ and $\mathcal{F}^2_{\lambda}$, to $\mathcal{B}$ and $\mathcal{B}^*$, respectively.} 
\end{proof}

\subsection{Dependence of the limiting process}	\label{sec:dependence1}
	
 As previously mentioned, the process
$(\mathcal{B},\mathcal{B}^*)$ 
is not Gaussian\rouge{. I}n this section we examine its  dependency structure.
%

	%
\begin{theorem}
	\label{th:correlation1}
The correlation coefficient function 
\[
	\mathrm{Corr}(t) = \frac{\mathds{E}\left[B(t)B^*(t)\right]-\mathds{E}[B(t)]\mathds{E}[B^*(t)]}{\sqrt{\mbox{\emph{Var}}(B(t))}\sqrt{\mbox{\emph{Var}}(B^*(t))}} = \frac{\mathds{E}\left[B(t)B^*(t)\right]}{t} \quad \mbox{for $t > 0$.} 
	\]
of $\mathcal{B}$ and $\mathcal{B}^*$ is given by 
\begin{align}
	\label{eq:corrf1}
 \mathrm{Corr}(t) = \frac{\alpha^{-1} - \beta^{-1}}{\alpha^{-1}+\beta^{-1}} + \frac{2\alpha(1-e^{-t(\alpha+\beta)})}{t(\alpha+\beta)^2}.
\end{align}
\end{theorem}
\begin{proof} Fix $t\ge 0$. Let $M_s:=B({t\wedge s})$ and $N_s:=B^*({t\wedge s})$ for all $s\ge 0$. The processes $\{M_s\}_{s\ge 0}$ and $\{N_s\}_{s\ge 0}$ are martingales bounded in $L_2$, so that \cite[Prop.~4.15]{legall2018brownian}
\[
	\mathds{E}\big[B(t)B^*(t)\big] = \mathds{E}[M_\infty N_\infty] = \mathds{E}[\langle M, N\rangle_\infty],\]
where $\{\langle M, N\rangle_s\}_{0\le s}$ denotes the bracket process of $\{M^t_s\}_{s\ge 0}$ and $\{N^t_s\}_{s\ge 0}$ \cite[Defn~4.14]{legall2018brownian}. Since $\{N_s\}_{s\ge 0}$ admits the stochastic integral representation (\ref{eq:defBstar}), then {by}~\cite[Thm~5.4]{legall2018brownian}
\begin{align*}
\mathds{E}[\langle M, N\rangle_\infty] = \mathds{E}\left[\int_{0}^t (-1)^{X(s)}\dd s\right] = \int_0^t \mathds{E}\left[(-1)^{X(s)}\right]\dd s. 
\end{align*}

Straightforward computations show that
\begin{align*}
\mathds{P}(X(s)=0)=\left([1,0]\exp\left( \Pi s\right)\right)_1 &= \frac{\alpha^{-1}}{\alpha^{-1}+\beta^{-1}} + \frac{\alpha}{\alpha+\beta}e^{-s(\alpha+\beta)},\\
\mathds{P}(X(s)=1)=\left([1,0]\exp\left( \Pi s\right)\right)_2 &=\frac{\beta^{-1}}{\alpha^{-1}+\beta^{-1}} - \frac{\alpha}{\alpha+\beta}e^{-s(\alpha+\beta)}.
\end{align*}
{T}hus, 
\begin{align*}
\mathds{E}\big[B(t)B^*(t)\big]& = \int_0^t\mathds{E}\left((-1)^{X(s)}\right)\dd s\\
& = \int_0^t \frac{\alpha^{-1} - \beta^{-1}}{\alpha^{-1}+\beta^{-1}} + \frac{2\alpha}{\alpha+\beta}e^{-s(\alpha+\beta)}\dd s\\
& =  t \frac{\alpha^{-1} - \beta^{-1}}{\alpha^{-1}+\beta^{-1}} + \frac{2\alpha(1-e^{-t(\alpha+\beta)})}{(\alpha+\beta)^2}.
\end{align*}
\end{proof}
%


\begin{Remark} We observe from Theorem \ref{th:correlation1} that the
    limiting process $(\mathcal{B},\mathcal{B}^*)$ has a time-dependent
    correlation function, which starts in $1$, is strictly decreasing
    and converges to
    $(\alpha^{-1}-\beta^{-1})/(\alpha^{-1}+\beta^{-1})$ as
    $t\rightarrow\infty$. From a modelling perspective, this provides
    an alternative to classic correlated bivariate Brownian motion
    models, for which the correlation function remains constant over
    time.
\end{Remark}

A{n} analogous construction can be made so that $\mathcal{B}^*$ starts in a desynchronized environment, that is, $\mathds{P}(X(0)=1)=1$. A slight modification to the proof of Theorem~\ref{th:correlation1} shows that the correlation coefficient for such a construction is given by
\[
	\mathrm{Corr}(t)=\frac{\alpha^{-1} - \beta^{-1}}{\alpha^{-1}+\beta^{-1}} - \frac{2\beta(1-e^{-t(\alpha+\beta)})}{t(\alpha+\beta)^2}.
\]

\section{Reflection at MAP times}\label{sec:MAP}
In this section, we construct a flip-flop approximation to a new class
of bivariate Brownian motion processes with a more flexible
time-dependent correlation function, which we call
\emph{two-dimensional MAP alternating Brownian motion}. They are a
generalization of the process in Section~\ref{sec:exprefl}
 as the synchronisation and desynchronisation epochs
now occur at arrival epochs of a continuous-time MAP. 
%
In continuous time, a MAP of parameters $(\bm{b},C, D)$ is a
counting process $\mathcal{K}$ driven by an underlying Markov jump
process with initial distribution $\bm{b}$ and intensity matrix
$C+D$. The epochs of increase of
$\mathcal{K}=\{K(t)\}_{t \geq 0}$ coincide with those jumps epochs
which occur due to the intensities of the matrix $D$. 
See \cite{latouche1999introduction} for a detailed exposition
on the topic.

\begin{definition} 
	Let $\mathcal{K} = \{K(t)\}_{t \geq 0}$ be a continuous-time MAP$(\bm{b},C,D)$, and $\mathcal{B}$ be a standard Brownian motion. Define $\mathcal{B}^*=\{B^*(t)\}_{t\ge 0}$ such that 
\[
B^*(t) = \int_0^t (-1)^{K(s)}\dd B_s, \quad t \geq 0. 
\]
The process $(\mathcal{B}, \mathcal{B}^*)$ is called a two-dimensional MAP alternating Brownian motion.
\end{definition} 

\subsection{Construction} 
	\label{subsec:constructMAP}  
    Here, let $\mathcal{J}_\lambda^1 = \mathcal{J}_\lambda$ where $\mathcal{J}_\lambda$ is defined as in \eqref{eqn:Jlambda}, and let $\mathcal{J}_\lambda^2=\{J_\lambda^2 (t)\}_{t\ge 0}$ where
  	\begin{align*} 
		{J_\lambda^2 (t) := (-1)^{K(t)}J_\lambda (t),\quad t\ge 0.}
	\end{align*} 
Define the flip-flop processes $\mathcal{F}^{1}_\lambda=\{F_\lambda^1 (t)\}_{t\ge 0}$ and $\mathcal{F}^2_\lambda=\{F_\lambda^2 (t)\}_{t\ge 0}$ by
	\begin{align*} 
		F^{i}_\lambda(t) := \sqrt{\lambda}\int_0^t \pi^*_i(\bm{J}_\lambda (s))\,\dd s,\quad i\in\{1,2\},\,t\ge 0,
	\end{align*}
where
$\bm{\mathcal{J}}_\lambda=(\mathcal{J}_\lambda^1,\mathcal{J}_\lambda^2,
  \mathcal{K})$, and 
	\begin{align*} 
		\pi_1^*((1,\mathcal{E})\times\mathcal{S}) & =\pi_2^*((\mathcal{E},1)\times\mathcal{S})=1, \\
		\pi_1^*((-1,\mathcal{E})\times\mathcal{S}) & =\pi_2^*((\mathcal{E},-1)\times\mathcal{S})=-1. 
	\end{align*} 
	 {The process $\bm{\mathcal{J}}_\lambda$ is a Markov jump
           process with state space $\mathcal{E}^2\times \mathcal{S}$,
           whose distribution is described in Theorem \ref{t:bivint}
           below. 
}

\begin{theorem} 
   \label{t:bivint}
The phase process $\boldsymbol{\mathcal{J}}_{\lambda} = (\mathcal{J}^1_{\lambda},\mathcal{J}^2_{\lambda},
  \mathcal{K})$, which lives on 
space $\mce^2\times\mathcal{S}$, has the initial distribution $(1/2\bm{b},\bm{0},\bm{0}, 1/2\bm{b})$ and the intensity matrix
\begin{align}
	\label{eq:bivintmatr3}
Q_\lambda =	%
\left[\begin{array}{cccc} 
		-\lambda I +C & D & 0 & \lambda I \\
D & -\lambda I + C & \lambda I & 0\\
0 & \lambda I &-\lambda I + C& D\\
\lambda I& 0 & D& -\lambda I + C\
	\end{array} \right],
	\end{align}
where $I$ is an $|\mathcal{S}| \times |\mathcal{S}|$ identity matrix. 
\end{theorem}
\begin{proof}
Fix $j,k\in\{1,-1\}$ and $t\ge 0$, and suppose that $(J^1_\lambda(t),
J^2_\lambda(t) ,
  K(t))$  is in some state in $(j,k)\times\mathcal{S}$. In {that case},
there are three types of possible jumps that can occur immediately
after $t$, which we explain next. 
\begin{itemize}
  \item Jump that stays within $(j,k)\times\mathcal{S}$: This jump does not cause a change in direction of $F^1_\lambda({t+})$ nor $F^2_\lambda({t+})$. It is associated {with} a \emph{hidden} jump of the MAP $\mathcal{K}$, which occurs according to the {non-diagonal} intensities of $C$. 
  \item Jumps to $(j,-k)\times\mathcal{S}$: In this case, a {synchronisation} or desynchronisation event occurs, causing a change in direction for $F^2({t+})$. This happens according to the intensities of the matrix $D$.
  \item Jumps to $(-j,-k)\times\mathcal{S}$: Both $F^1_\lambda({t+})$
    and $F^2_\lambda({t+})$ change direction, which is associated to a
    jump epoch of $\mathcal{J}_\lambda$ that occurs with intensity
    $\lambda$. The matrix $\lambda I$ guarantees that for a jump of
    this type, the second coordinate in $\mathcal{S}$ stays fixed.
\end{itemize}
Finally, {note} that a jump to $(-j,k)\times\mathcal{S}$ is imposible; {this is} because the collection of epochs at which $\mathcal{F}^1_\lambda$ changes direction is a subset of the collection of epochs at which $\mathcal{F}^2_\lambda$ changes direction.
\end{proof} 

\begin{theorem}
	\label{th:relfectingmap1}
The bivariate flip-flop $(\mathcal{F}^1_{\lambda},
\mathcal{F}^2_{\lambda})$ with the underlying process
$(\mathcal{J}^1_{\lambda},\mathcal{J}^2_{\lambda})$ converges
strongly to $(\mathcal{B},\mathcal{B}^*)$ as $\lambda \rightarrow \infty$. 
\end{theorem}
{The proof of this theorem} follows verbatim that of Theorem
\ref{th:relfectingexp1strong} by replacing $\mathcal{X}$ with
$\mathcal{K}$, and {noting} that the number of jumps of $\mathcal{K}$
in $[0,T]$ has a Poissonian tail with rate $\gamma_1T$ where $\gamma_1:=\max_{i\in \mathcal{S}}|C_{ii}|$. 

\subsection{Dependence of the limiting process} 

We assess the correlation between $\mathcal{B}$ and
$\mathcal{B}^*$. {Though the analysis is similar to that of Section \ref{sec:dependence1}, here we use matrix-analytic methods to find a closed-form formula.} 
%
\begin{theorem}
	\label{th:correlationMAP1}
Let $(\mathcal{B}, \mathcal{B}^*)$ be a two-dimensional MAP
alternating Brownian motion. The correlation coefficient function of $\mathcal{B}$ and $\mathcal{B}^*$ is given by
 \[
 	\mathrm{Corr}(t) = \frac{1}{t}\left[\begin{array}{ccc}\bm{b}&\bm{0}& 0\end{array}\right] \exp\left(\left[\begin{array}{rrr} C&D&\bm{e}\\ D&C& -\bm{e}\\ 0&0&0\end{array}\right] t\right) \left[\begin{array}{c}\bm{0}\\\bm{0}\\ 1\end{array}\right],
\]
where $\bm{e}$ denotes a column vector of $1$s of appropriate size.
\end{theorem}
\begin{proof}
{By analogous arguments to those} in the proof of Theorem \ref{th:correlation1}, we have
\begin{align}
	\label{eq:btbtstarMAP}
	\mathds{E}\big[B(t)B^*(t)\big]= \int_0^t\mathds{E}\left[(-1)^{K(s)}\right]\dd s.\end{align} 

Thus, what remains is to compute the transition probabilities of $\{K(s)\}_{s\ge 0}$, and {to} find a closed form formulae for the integral expression in the RHS of (\ref{eq:btbtstarMAP}). First,
\begin{align*}
\mathds{P}(K(s)=0)&=\left[\begin{array}{cc}\bm{b}&\bm{0}\end{array}\right] \exp\left(\left[\begin{array}{cc}C&D\\ D& C\end{array}\right]s\right) \left[\begin{array}{c}\bm{e}\\\bm{0}\end{array}\right]\quad \mbox{and}\\
\mathds{P}(K(s)=1)&=\left[\begin{array}{cc}\bm{b}&\bm{0}\end{array}\right] \exp\left(\left[\begin{array}{cc}C&D\\ D& C\end{array}\right]s\right) \left[\begin{array}{c}\bm{0}\\\bm{e}\end{array}\right],
\end{align*}
so that
\[\mathds{E}\left[(-1)^{K(s)}\right] =\left[\begin{array}{cc}\bm{b}&\bm{0}\end{array}\right] \exp\left(\left[\begin{array}{cc}C&D\\ D& C\end{array}\right]s\right) \left[\begin{array}{r}\bm{e}\\-\bm{e}\end{array}\right].\]
Using \cite[Thm 1]{van1978computing}, we have
\[\exp\left(\left[\begin{array}{rr|r} C&D&\bm{e}\\ D&C& -\bm{e}\\\hline 0&0&0\end{array}\right] t\right)=\left[\begin{array}{c|c} \exp\left(\left[\begin{array}{rr} C&D\\ D&C\end{array}\right] t\right)& \displaystyle\int_0^t \exp\left(\left[\begin{array}{rr} C&D\\ D&C\end{array}\right] s\right)\left[\begin{array}{r}\bm{e}\\-\bm{e}\end{array}\right]\dd s\\&\vspace*{-0.4cm}\\\hline 0 &1\end{array}\right],\]
implying
\[
	\int_0^t\mathds{E}\left((-1)^{K(s)}\right)\dd s = \left[\begin{array}{ccc}\bm{b}&\bm{0}& 0\end{array}\right] \exp\left(\left[\begin{array}{rrr} C&D&\bm{e}\\ D&C& -\bm{e}\\ 0&0&0\end{array}\right] t\right) \left[\begin{array}{c}\bm{0}\\\bm{0}\\ 1\end{array}\right]\]
and the result follows.
\end{proof}

Note that the matrix $\left[\begin{smallmatrix}C&D\\ D&
    C\end{smallmatrix}\right]$ is an intensity matrix, and thus
singular. This means the identity $\int_0^t e^{As}\dd
s=(A)^{-1}(e^{At}-I)$ could not be used for
$A=\left[\begin{smallmatrix}C&D\\ D& C\end{smallmatrix}\right]$ in the
proof of Theorem \ref{th:correlationMAP1}.

\begin{Remark}  {We can recover the formula in Theorem \ref{th:correlation1} by letting
\begin{align*} 
	\bm{b}=(1/2,1/2),\quad C= \left[\begin{array}{rr} -			\alpha & 0\\ 
		0&-\beta
	\end{array}\right], \quad D= \left[\begin{array}{rr} 0& \alpha\\ \beta & 0\end{array}\right].
\end{align*} 
This
would give us a bivariate flip-flop with 8 states, {partitioned into two communication classes of size 4 with the same transition rates, and which can be identified with a 4-state
model instead.}}
\end{Remark}

\section*{Acknowledgements.}
The second and third authors are affiliated with Australian Research Council (ARC) Centre of Excellence for Mathematical and Statistical Frontiers (ACEMS). We would like to also acknowledge the support of the ARC DP180103106 grant.

\bibliographystyle{abbrv}
\end{document}